\RequirePackage{ifpdf}
\ifpdf 
\documentclass[pdftex]{UNIM}
\else
\documentclass{UNIM}
\fi

\usepackage{amsmath}
\usepackage{amsthm} 
\usepackage{amsfonts} 
\usepackage{amssymb} 
\usepackage{graphicx}
\usepackage{appendix}
\usepackage{hyperref}
\usepackage{bbm}
\usepackage{mathrsfs}
\usepackage{pgf,tikz}
\usepackage{mathrsfs}
\usepackage{stmaryrd}
\usepackage{enumitem}
\usepackage{multicol}
\usepackage{parcolumns}
\usepackage[all]{xy}
\usetikzlibrary{arrows}

\numberwithin{equation}{section}

\DeclareMathAlphabet{\mathpzc}{OT1}{pzc}{m}{it}

\begin{document}
\thispagestyle{empty}
\ArticleName{Unimodularity Criteria for Poisson Structures on Foliated Manifolds}
\ShortArticleName{Unimodularity Criteria for Poisson Structures on Foliated Manifolds}
\Author{Andr\'es Pedroza$^{1}$, Eduardo Velasco-Barreras$^{2}$, and Yury Vorobiev$^{3}$.}
\AuthorNameForHeading{Andr\'es Pedroza, Eduardo Velasco-Barreras, and Yury Vorobiev.}

\noindent\emph{Faculty of Science, University of Colima}

\noindent\emph{Bernal D\'iaz del Castillo 340, Colima, M\'exico, 28045.}$^{1}$\vspace{6pt}

\noindent\emph{Department of Mathematics, University of Sonora}

\noindent\emph{Rosales y Blvd. Luis Encinas, Hermosillo, M\'exico, 83000.}$^{2,3}$\vspace{6pt}

\noindent E-mail adresses: \emph{\href{mailto:andres_pedroza@ucol.mx}{andres\_pedroza@ucol.mx}$^{1}$,\href{mailto:lalo.velasco@mat.uson.mx}{lalo.velasco@mat.uson.mx}$^{2}$,\href{mailto:yurimv@guaymas.uson.mx}{yurimv@guaymas.uson.mx}$^{3}$}.

\Abstract{We study the behavior of the modular class of an orientable Poisson manifold and formulate some unimodularity criteria in the semilocal context, around a (singular) symplectic leaf. In particular, we show that the unimodularity of the transverse Poisson structure of the leaf is a necessary condition for the semilocal unimodular property. Our main tool is an explicit formula for a bigraded decomposition of modular vector fields of a coupling Poisson structure on a foliated manifold.}
\Classification{ 53C05; 53C12; 53D17.}

\section{Introduction}\label{intro}

The modular class $\operatorname{Mod}(M,\Pi)$ of an orientable Poisson manifold $(M,\Pi)$ is a distinguished element of the first Poisson cohomology group $H_{\Pi}^{1}(M)$ and gives an obstruction to the existence of a volume form which is invariant under the flow of every Hamiltonian vector field \cite{Koszul,We-97}. If the modular class is trivial, then such an invariant volume form exists and the Poisson manifold is said to be unimodular.

In the regular case, when the rank of the Poisson tensor $\Pi$ is locally constant, we have the following fact \cite{We-97,AB-03}: the modular class $\operatorname{Mod}(M,\Pi)$ is equivalent to the Reeb class $\operatorname{Mod}(\mathcal{S})$ of the regular symplectic foliation $\mathcal{S}$ of $\Pi$ (for the case $\operatorname{codim}\mathcal{S}=1$, see \cite{GMP-2011}). For a transversally orientable regular foliation, the Reeb class is the obstruction to the existence of a closed transversal volume element \cite{Hur-86}. This relationship leads to a geometric criterion: the triviality of the Reeb class of $\mathcal{S}$ is equivalent to the unimodularity of the regular Poisson manifold $(M,\Pi)$ (see also \cite{Cra-03}). As a consequence, the unimodularity of a regular Poisson manifold only depends on its characteristic (symplectic) foliation rather than the leaf-wise symplectic form. Along with the standard approaches \cite{Hur-86,AB-03}, one can characterize the Reeb class in different ways, for example, by using the Bott connection \cite{We-97} or, as the modular class of the associated Lie algebroid \cite{ELW,KS-00}.

In this paper, we are interested in a generalization of these results to the case of Poisson manifolds with singular symplectic foliations, for which it does not exist a direct analog of the Reeb class. For instance, some necessary conditions of unimodularity can be derived from the relationship between the modular class and the linear Poisson holonomy introduced in \cite{GiGo-2001}. Our goal is to study the behavior of the modular class of an orientable Poisson manifold $(M,\Pi)$ and formulate some unimodularity criteria in the semilocal context, around a (singular) symplectic leaf $S$. The semilocal Poisson geometry is related to the study of the so-called coupling Poisson structures on fibered or, more generally, foliated manifolds \cite{Vo-01,Va-04}. As is known \cite{Vo-01}, the Poisson structure near an embedded symplectic leaf $S$ is realized as a coupling Poisson structure. In particular, this fact gives rise to the notion of a transverse Poisson structure $P$ of the leaf $S$ (if $S$ is regular, then $P\equiv0$). Therefore, we address the question on the study of modular classes to the class of coupling Poisson structures.

Due to local Weinstein's splitting theorem \cite{We-83}, the unimodularity of $\Pi$ in a neighborhood of a singular point is provided by the unimodularity of the transverse Poisson structure of the point. In the nonzero dimensional case, we describe some obstructions to the semilocal unimodularity of the leaf which are related to some ``tangential'' and ``transversal'' characteristics of $S$. In particular, we show that the unimodularity of a transverse Poisson structure $P$ of $S$ is a necessary condition for
$\operatorname{Mod}(M,\Pi)=0$ (Proposition \ref{prop:LeafMods}). Moreover, we prove that under the vanishing of the modular class of $P$, some cohomological obstructions possibly appear in the first cohomology of the associated cochain complex \cite{VeVo-16,Vo-05} (Theorem \ref{teo:Unimodular3}). In the case when the neighborhood of the leaf is ``flat'', these obstructions are directly related to the Reeb class of a foliation (Theorem \ref{teo:UnimodFlat}). In particular, this occurs in the regular case.

Our main results are based on an explicit formula for a bigraded decomposition of the modular vector fields of a coupling Poisson structure $\Pi$ on a foliated manifold (Proposition \ref{prop:CoupMod}). This formula involves the modular vector field of the Poisson foliation associated to $\Pi$, which is related to the Reeb class (Proposition \ref{prop:ModsAndReeb}) and whose foliated Poisson cohomology class can be interpreted in terms of a more general notion of the modular class of a triangular Lie bi-algebroid \cite{KS-00}. Also, we study the behavior of the unimodularity property under gauge equivalence \cite{SeWe-01,Br-05} (Proposition \ref{prop:Gauge}). A similar problem for the Morita equivalence of Poisson structures was studied in \cite{GiGo-2001,Cra-03}.

This paper is organized as follows. In Section \ref{sec:OriFol} we review some basic notions and facts about foliations. In Section \ref{sec:FolMod} we study the relationship between the modular class of a Poisson foliation and the Reeb class. The modular vector fields of coupling Poisson structures are described in Section \ref{sec:CoupMod} by using bigraded calculus on foliated manifolds. Section \ref{sec:gauge} is devoted to the study of the behavior of the unimodularity property under gauge transformations. In Section \ref{sec:Unimod} we derive some unimodularity criteria for coupling Poisson structures and find cohomological obstructions to the unimodularity. In Section \ref{sec:Flat} we examine the general unimodularity criteria for compatible Poisson structures on flat Poisson foliations. Here the cohomological obstructions take values on the foliated de Rham-Casimir complex. Finally, in Section \ref{sec:Leaf} we apply the above results to describe the unimodularity in a neighborhood of a symplectic leaf.

\section{Preliminaries: Orientable Foliations}\label{sec:OriFol}

We start by recalling some definitions and known facts about calculus on foliated manifolds (for details, we refer to \cite{Hur-86,KMS-93,Va-04}).

\paragraph{Foliated de Rham Differential.} Let $\mathcal{V}$ be a regular foliation on $M$. Denote by $\mathbb{V}:=T\mathcal{V}$ the tangent bundle of $\mathcal{V}$. There exists a derivation $\operatorname{d}_{\mathcal{V}}\in \operatorname{Der}_{\mathbb{R}}^{1}\Gamma(\wedge^{\bullet}\mathbb{V}^{*})$ of degree $1$ which is a coboundary operator, $\operatorname{d}_{\mathcal{V}}^{2}=0$, called the \emph{foliated exterior derivative}. Notice that the foliated de Rham complex $(\Gamma(\wedge^{\bullet}\mathbb{V}^{*}),\operatorname{d}_{\mathcal{V}})$ is just the cochain complex of the Lie algebroid $\mathbb{V}$ associated to the foliation $\mathcal{V}$. The cohomology of the foliated de Rham complex of $\mathcal{V}$ will be denoted by $H_{\operatorname{dR}}^{\bullet}(\mathcal{V})$.

Fix a normal distribution $\mathbb{H}\subset TM$ of the foliation $\mathcal{V}$,
\begin{equation}
TM=\mathbb{H}\oplus\mathbb{V}. \label{eq:SplitHV0}
\end{equation}
Then, the vector-valued 1-form $\gamma\in\Gamma(T^{*}M\otimes\mathbb{V})$, defined as the projection $\operatorname{pr}_{\mathbb{V}}:TM\rightarrow\mathbb{V}$ along $\mathbb{H}$ in \eqref{eq:SplitHV0}, is said to be a \emph{connection form} on the foliated manifold $(M,\mathcal{V})$. Conversely, every vector-valued 1-form $\gamma\in\Gamma(T^{*}M\otimes\mathbb{V})$ with $\gamma|_{\mathbb{V}}=\operatorname{Id}_{\mathbb{V}}$ induces the normal bundle $\mathbb{H}:=\ker\gamma$ of $\mathcal{V}$. Then, the curvature form $R^{\gamma}\in\Gamma(\wedge^{2}T^{*}M\otimes\mathbb{V})$ of the connection is given by \cite{KMS-93}
\[
R^{\gamma}(X,Y):=\gamma[(\operatorname{Id}_{TM}-\gamma)X,(\operatorname{Id}_{TM}-\gamma)Y] \qquad\forall X,Y\in\Gamma(TM)
\]
and controls the integrability of the normal bundle $\mathbb{H}$. The connection $\gamma$ is said to be \emph{flat} if $R^{\gamma}=0$.

Splitting \eqref{eq:SplitHV0} induces $\mathbb{H}$-dependent bigradings of the exterior algebras of multivector fields and differential forms on $M$:
\begin{equation}\label{Big}
\Gamma(\wedge^{\bullet}TM)=\bigoplus_{p,q\in\mathbb{Z}}\Gamma(\wedge^{p,q}TM), \qquad \Gamma(\wedge^{\bullet}T^{*}M)=\bigoplus_{p,q\in\mathbb{Z}}\Gamma(\wedge^{p,q}T^{*}M).
\end{equation}
Here, $\wedge^{p,q}TM:=\wedge^{p}\mathbb{H}\otimes\wedge^{q}\mathbb{V}$ and $\wedge^{p,q}T^{*}M:=\wedge^{p}\mathbb{V}^{0}\otimes\wedge^{q}\mathbb{H}^{0}$, where $\mathbb{H}^{0}:=\operatorname{Ann}(\mathbb{H})$ and $\mathbb{V}^{0}:=\operatorname{Ann}(\mathbb{V})$ denote the annihilators of $\mathbb{H}$ and $\mathbb{V}$, respectively. For a multivector field $A$, the term of bidegree $(p,q)$ in decomposition \eqref{Big} is denoted by $A_{p,q}$. We follow same notation for differential forms. Moreover, we have a bigraded decomposition for any linear operator on these exterior algebras. In particular, the exterior differential splits as $\operatorname{d}=\operatorname{d}_{1,0}^{\gamma}+\operatorname{d}_{2,-1}^{\gamma}+\operatorname{d}_{0,1}^{\gamma}$, where $\operatorname{d}_{1,0}^{\gamma}$ is the \emph{covariant exterior derivative} of $\gamma$ and $\operatorname{d}_{2,-1}^{\gamma}=-\mathbf{i}_{R^{\gamma}}$. Furthermore,
\begin{equation}\label{eq:LieCurv}
(\operatorname{d}_{1,0}^{\gamma})^{2}=\operatorname{L}_{R^{\gamma}}
\end{equation}
and $(\operatorname{d}_{0,1}^{\gamma})^{2}=0$ (for the definition of the Lie derivative $\operatorname{L}_{R^{\gamma}}$, see \cite{KMS-93}). It is clear that the canonical inclusion of the leaves of $\mathcal{V}$ in $M$ induces a cochain complex isomorphism,
\begin{equation}\label{eq:IsoDeRhamFol}
\left(\Gamma(\wedge^{\bullet}\mathbb{V}^{*}),\operatorname{d}_{\mathcal{V}}\right)\cong\left(\Gamma(\wedge^{\bullet}\mathbb{H}^{0}),\operatorname{d}_{0,1}^{\gamma}\right).
\end{equation}
For each $\mu\in\Gamma(\wedge^{\bullet}\mathbb{V}^{*})$, we will denote by $\mu_{\gamma}\in\Gamma(\wedge^{\bullet}\mathbb{H}^{0})$ the corresponding element under the above isomorphism. We use the same notation for cohomology classes.

We will denote by
\[
\mathfrak{X}_{\mathrm{pr}}(M,\mathcal{V}) := \{X\in\mathfrak{X}(M)\mid[X,\Gamma(\mathbb{V})]\subset\Gamma(\mathbb{V})\}
\]
the Lie subalgebra of $\mathcal{V}$-\emph{projectable vector fields}. For each $\mu\in\Gamma(\wedge^{\bullet}\mathbb{V}^{*})$ and $X\in
\mathfrak{X}_{\mathrm{pr}}(M,\mathcal{V})$, the Lie derivative $\operatorname{L}_{X}\mu\in\Gamma(\wedge^{\bullet}\mathbb{V}^{*})$ is well-defined by the
standard formula.

\paragraph{Divergence 1-Form.} Suppose that the foliation $\mathcal{V}$ on $M$ is \emph{orientable}, that is, there exists a nowhere vanishing element $\tau\in\Gamma(\wedge^{\operatorname{top}}\mathbb{V}^{*})$, called a \emph{leaf-wise volume form} of $\mathcal{V}$. Therefore, the restriction of $\tau$ to each leaf $L$ of $\mathcal{V}$ gives a volume form on $L$. For each $X\in\mathfrak{X}_{\mathrm{pr}}(M,\mathcal{V})$, the divergence $\operatorname{div}^{\tau}(X)\in C^{\infty}(M)$ with respect to $\tau$ is defined by the relation $\operatorname{L}_{X}\tau=\operatorname{div}^{\tau}(X)\tau$.

Fix a connection form $\gamma$ on $(M,\mathcal{V})$ associated to a normal bundle $\mathbb{H}$ of $\mathcal{V}$. Then, one can think of a leaf-wise volume form $\tau\in\Gamma(\wedge^{\operatorname{top}}\mathbb{V}^{*})$ as a differential form $\tau_{\gamma}\in\Gamma(\wedge^{\operatorname{top}}\mathbb{H}^{0})$ vanishing only on the sections of
$\mathbb{H}$. Recall that $\mathbb{H}^{0}$ and $\mathbb{V}$ are of the same rank $k$. The divergence is given by the formula
\begin{equation}
\operatorname{div}^{\tau}(X)\tau_{\gamma}=(\operatorname{L}_{X}\tau_{\gamma})_{0,k} \label{eq:DivConnection}
\end{equation}
for any $X\in\mathfrak{X}_{\mathrm{pr}}(M,\mathcal{V})$. Here, the bigraded decomposition of the $k$-form $\operatorname{L}_{X}\tau_{\gamma}$ consists of the terms of bidegree $(0,k)$ and $(1,k-1)$.

Now, we observe that there exists a unique 1-form vanishing on vector fields tangent to $\mathcal{V}$, $\theta_{\tau}^{\gamma}\in\Gamma(\mathbb{V}^{0})$, and such that
\begin{equation}\label{eq:DefTheta}
\operatorname{d}_{1,0}^{\gamma}\tau_{\gamma}=\theta_{\tau}^{\gamma}\wedge\tau_{\gamma}.
\end{equation}
Denote by $\Gamma_{\mathrm{pr}}(\mathbb{H}) := \mathfrak{X}_{\mathrm{pr}}(M,\mathcal{V}) \cap \Gamma(\mathbb{H})$ the set of all projectable sections of $\mathbb{H}$. Then, $\theta_{\tau}^{\gamma}$ is related with the divergence by the condition
\begin{equation}\label{eq:Theta1}
\theta_{\tau}^{\gamma}(X)=\operatorname{div}^{\tau}(X)\qquad\forall X\in\Gamma_{\mathrm{pr}}(\mathbb{H}).
\end{equation}
Therefore, the 1-form $\theta_{\tau}^{\gamma}\in\Gamma(\mathbb{V}^{0})$ can be called the \emph{divergence form} associated to the pair $(\tau,\gamma)$. By using \eqref{eq:DefTheta}, \eqref{eq:LieCurv}, and \eqref{eq:DivConnection}, one can derive the following useful relations
\begin{align}
\operatorname{d}_{1,0}^{\gamma}\theta_{\tau}^{\gamma}(X_{1},X_{2}) &= \operatorname{div}^{\tau}(R^{\gamma}(X_{1},X_{2})),\label{eq:Theta2}\\
\theta_{f\tau}^{\gamma}&=\theta_{\tau}^{\gamma}+\operatorname{d}_{1,0}^{\gamma}\ln|f|,\label{eq:Theta3}
\end{align}
for all $X_{1},X_{2}\in\Gamma_{\mathrm{pr}}(\mathbb{H})$ and each nowhere vanishing $f\in C^{\infty}(M)$.

\paragraph{The Reeb Class.} Let $\mathcal{V}$ be a regular foliation of $M$. Consider the tangent bundle $\mathbb{V}=T\mathcal{V}$ and its annihilator $\mathbb{V}^{0}$. We say that the foliation $\mathcal{V}$ is \emph{transversally orientable} if there exists a nowhere vanishing element $\varsigma\in\Gamma(\wedge^{\operatorname{top}}\mathbb{V}^{0})$. In this case, we say that $\varsigma$ is a \emph{transversal volume element} of $\mathcal{V}$. In particular, we have $\mathbb{V}=\{X\in TM \mid\mathbf{i}_{X}\varsigma=0\}$.

It follows from the identity $\mathbf{i}_{[X,Y]}=[\operatorname{L}_{X},\mathbf{i}_{Y}]$ $\forall X,Y\in\Gamma(\mathbb{V})$ that the Lie derivative along every $\mathcal{V}$-tangent vector field preserves the space of sections
$\Gamma(\wedge^{\operatorname{top}}\mathbb{V}^{0})$. As a consequence, for each transversal volume element $\varsigma$ of $\mathcal{V}$, there exists a unique foliated 1-form $\lambda_{\varsigma}\in\Gamma(\mathbb{V}^{*})$ defined by the relation
\begin{equation}\label{eq:Reeb0}
\operatorname{L}_{X}\varsigma=\lambda_{\varsigma}(X)\varsigma, \qquad\forall X\in\Gamma(\mathbb{V}).
\end{equation}
Then, $\lambda_{\varsigma}$ is a closed foliated 1-form, $\operatorname{d}_{\mathcal{V}}\lambda_{\varsigma}=0$. Moreover, by the standard arguments, the $\operatorname{d}_{\mathcal{V}}$-cohomology class (foliated de Rham cohomology class)
\begin{equation}\label{eq:Reeb}
\operatorname{Mod}(\mathcal{V}):=[\lambda_{\varsigma}]\in H_{\operatorname{dR}}^{1}(\mathcal{V}),
\end{equation}
is independent of the choice of a transversal volume element $\varsigma$ and called the \emph{Reeb class} of $\mathcal{V}$ (see, for example, \cite{Hur-86,Hur-02,AB-03}). The Reeb class is an obstruction to the existence of a transversal volume element of $\mathcal{V}$ which is invariant under the flow of any vector field tangent to the foliation. Alternatively, the vanishing of $\operatorname{Mod}(\mathcal{V})$ is equivalent to the existence of a closed transversal volume element of $\mathcal{V}$ \cite{Hur-86}.

\begin{example}\label{ex:FiberReeb}
Let $\pi:M\rightarrow S$ be a fiber bundle over an orientable base $S$. Consider the foliation $\mathcal{V}:=\{\pi^{-1}(x)\}_{x\in S}$ on $M$ given by the surjective submersion $\pi$, called \emph{simple foliation}. Then, $\mathcal{V}$ is a transversally orientable foliation with trivial Reeb class. Indeed, given a volume form $\varsigma_{0}$ on the base $S$, we get the transversal volume element $ \varsigma=\pi^{*}\varsigma_{0}$ of $\mathcal{V}$. It is clear that $\varsigma$ is closed on $M$ and hence $\operatorname{Mod}(\mathcal{V})=0$.
\end{example}

Pick a connection $\gamma$ on $(M,\mathcal{V})$ associated to a normal bundle $\mathbb{H}$ of $\mathcal{V}$. For each transversal volume element $\varsigma$ of $\mathcal{V}$, there exists a 1-form $\lambda_{\varsigma}^{\gamma}\in\Gamma(\mathbb{H}^{0})$ uniquely defined by the relation
\begin{equation}\label{R3}
\operatorname{d}_{0,1}^{\gamma}\varsigma=\lambda_{\varsigma}^{\gamma}\wedge\varsigma.
\end{equation}
From here, taking into account the bidegrees of $\lambda_{\varsigma}^{\gamma}$ and $\varsigma$, we conclude that $\operatorname{d}_{0,1}^{\gamma}\lambda_{\varsigma}^{\gamma}=0$ and hence $\lambda_{\varsigma}^{\gamma}$ is a 1-cocycle of $\operatorname{d}_{0,1}^{\gamma}$. Then, under the isomorphism \eqref{eq:IsoDeRhamFol}, the Reeb class of the foliation $\mathcal{V}$ equals the $\operatorname{d}_{0,1}^{\gamma}$-cohomology class of $\lambda_{\varsigma}^{\gamma}$,
\begin{equation}\label{R4}
\operatorname{Mod}(\mathcal{V})_{\gamma}=[\lambda_{\varsigma}^{\gamma}].
\end{equation}
Indeed, this is consequence of the following computation for all $X\in\Gamma(\mathbb{V})$:
\[
\lambda_{\varsigma}(X)\varsigma=\operatorname{L}_{X}\varsigma=\operatorname{d}\mathbf{i}_{X}\varsigma
+\mathbf{i}_{X}\operatorname{d}\varsigma=\mathbf{i}_{X}\operatorname{d}_{0,1}^{\gamma}\varsigma
=\lambda_{\varsigma}^{\gamma}(X)\varsigma.
\]
Observe that in the flat case, $R^{\gamma}=0$, each leaf-wise volume form $\tau\in\Gamma(\wedge^{\operatorname{top}}\mathbb{V}^{*})$ of $\mathcal{V}$ induces the transversal volume element $\tau_{\gamma}$ of the integral foliation $\mathcal{H}$ of $\mathbb{H}$. Furthermore, $\partial_{\mathcal{H}}:=\operatorname{d}_{1,0}^{\gamma}$ is the corresponding foliated exterior derivative and, by \eqref{eq:Theta2}, $\theta_{\tau}^{\gamma}$ is a 1-cocycle of $\partial_{\mathcal{H}}$. Then, taking into account \eqref{R4}, we conclude that the $\partial_{\mathcal{H}}$-cohomology class of $\theta_{\tau}^{\gamma}$ coincides with the Reeb class of $\mathcal{H}$.

To end this section we make the following remarks on the different interpretations of the Reeb class.

\begin{remark}
\bigskip The Reeb class is related to some characteristic classes of representations of the Lie algebroid $\mathbb{V}$ associated to the foliation $\mathcal{V}$ \cite{ELW,KS-08}. First, note that the Lie derivative along $\mathcal{V}$-tangent vector fields gives a representation $D$ on the line bundle $\wedge^{\operatorname{top}}\mathbb{V}^{0}$. By \eqref{eq:Reeb0} and \eqref{eq:Reeb}, the Reeb class is just the characteristic class of this representation, $\operatorname{Mod}(\mathcal{V})=\operatorname{Char}(\wedge^{\operatorname{top}}\mathbb{V}^{0})$. On the other hand, the Reeb class can be expressed in terms of the Bott connection $\nabla^{\operatorname{Bott}}$ on the normal bundle $E:=\frac{TM}{\mathbb{V}}$ of $\mathcal{V}$ \cite{We-97}. Under the natural identification $\mathbb{V}^{0}\cong E^{*}$, the dual of the representation $\nabla^{\operatorname{Bott}}$ on $\wedge^{\operatorname{top}}E$ coincides with $D$. Then, $\operatorname{Char}(\wedge^{\operatorname{top}}E)=-\operatorname{Mod}(\mathcal{V})$. Finally, we observe that the Reeb class coincides with the modular class of the Lie algebroid $\mathbb{V}$ \cite{ELW,KS-00,KS-08}, $\operatorname{Mod}(\mathcal{V}) = \operatorname{Char}(\wedge^{\operatorname{top}}\mathbb{V}\otimes\wedge^{\operatorname{top}}T^{*}M)$.
\end{remark}

\section{The Modular Class of a Poisson Foliation}\label{sec:FolMod}

In this section, we describe the relationship between the modular class of a leaf-tangent Poisson structure on a foliated manifold and the Reeb class.

First, let us recall the definitions and some properties of modular vector fields and the modular class of a Poisson manifold \cite{Koszul,We-97}. Let $(M,\Pi)$ be an orientable Poisson manifold with Poisson bivector field $\Pi$ on $M$. Denote by $\Pi^{\sharp}:T^{*}M\rightarrow TM$ the vector bundle morphism given by $\langle\beta,\Pi^{\sharp}\alpha\rangle:=\Pi(\alpha,\beta)$. Let $\operatorname{Poiss}(M,\Pi) := \{X\in\Gamma(TM)\mid L_{X}\Pi=0\}$ and $\operatorname{Ham}(M,\Pi) := \{\Pi^{\sharp}\operatorname{d}f\mid f\in C^{\infty}(M)\}$ be the Lie algebras of Poisson and Hamiltonian vector fields on $(M,\Pi)$, respectively. Then, the first Poisson cohomology is $H_{\Pi}^{1}(M) = \frac{\operatorname{Poiss}(M,\Pi)}{\operatorname{Ham}(M,\Pi)}$.

Given a volume form $\Omega$ of $M$, one can define a derivation $Z_{\Pi}^{\Omega}$ of $C^{\infty}(M)$ by the formula $Z_{\Pi}^{\Omega}(f) := \operatorname{div}^{\Omega} (\Pi^{\sharp}\operatorname{d}f)$, where $\operatorname{div}^{\Omega}$ is the divergence operator. The vector field $Z_{\Pi}^{\Omega}$ is a Poisson vector field of $\Pi$, called the \emph{modular vector field} \cite{Koszul,We-97} of the oriented Poisson manifold $(M,\Pi,\Omega)$.

In terms of the interior product, the modular vector field can be also defined by $-\mathbf{i}_{Z_{\Pi}^{\Omega}}\Omega=\operatorname{d}\mathbf{i}_{\Pi}\Omega$. Here, $\mathbf{i}$ denotes the insertion operator which on decomposable multivector fields is given by $\mathbf{i}_{X_{1}\wedge\ldots\wedge X_{k}}=\mathbf{i}_{X_{1}}\circ\ldots\circ\mathbf{i}_{X_{k}}$ $\forall X_{i}\in\Gamma(TM)$. Furthermore, if $\Omega^{\prime}$ is another volume form on $M$, then $Z_{\Pi}^{\Omega^{\prime}}=Z_{\Pi}^{\Omega}-\Pi^{\sharp}\operatorname{d}\ln|f|$, where $f=\Omega^{\prime}/\Omega$. Hence, the Poisson cohomology class of $Z_{\Pi}^{\Omega}$ is independent of the choice of $\Omega$. Therefore, $\operatorname{Mod}(M,\Pi) :=[Z_{\Pi}^{\Omega}]\in H_{\Pi}^{1}(M)$ is an intrinsic Poisson cohomology class called the \emph{modular class} \cite{We-97} of the orientable Poisson manifold $(M,\Pi)$. A Poisson manifold with vanishing modular class is said to be \emph{unimodular}. The modular class is an obstruction to the existence of a volume form which is invariant with respect to all Hamiltonian vector fields.

As an example, consider the 3-dimensional oriented (linear) Poisson manifold $(\mathbb{R}^{3}_{(x,y,z)},\Pi,\Omega)$, where $\Pi=\tfrac{1}{2}\tfrac{\partial}{\partial z}\wedge\left(x\tfrac{\partial}{\partial x}+y\tfrac{\partial}{\partial y}\right)$ and $\Omega$ is the Euclidean volume form. Then, $Z^{\Omega}_{\Pi}=\tfrac{\partial}{\partial z}$ cannot be a Hamiltonian vector field since it is non-zero at $0$. Moreover, on the regular domain $N_{\mathrm{reg}}:=\mathbb{R}^{3}-\{z\text{-axis}\}$, $\tfrac{\partial}{\partial z}$ admits the Hamiltonian $-\ln(x^{2}+y^{2})$. Thus, even though $\operatorname{Mod}(N_{\mathrm{reg}},\Pi|_{N_{\mathrm{reg}}})=0$, the Poisson manifold $(\mathbb{R}^{3}_{(x,y,z)},\Pi)$ is not unimodular.

\paragraph{Poisson Foliations and Orientability.} A \emph{Poisson foliation} consists of a triple $(M,\mathcal{V},P)$, where $\mathcal{V}$ is a regular foliation on a manifold $M$ endowed with a leaf-tangent Poisson bivector field $P\in\Gamma(\wedge^{2}\mathbb{V}),\mathbb{V}=T\mathcal{V}$. It is clear that the characteristic distribution of $P$ belongs to $\mathbb{V}$, $P^{\sharp}(T^{*}M)\subset\mathbb{V}$, and hence each leaf $L$ of $\mathcal{V}$ inherits from $P$ a unique Poisson structure $P_{L}$ such that the inclusion $\iota_{L}:L\hookrightarrow M$ is a Poisson map. Therefore, $M$ is foliated by the Poisson manifolds $(L,P_{L})$. Denote by $\operatorname{Poiss}_{\mathcal{V}}(M,P) := \Gamma(\mathbb{V}) \cap \operatorname{Poiss}(M,P)$ the Lie algebra of all Poisson vector fields of $P$ tangent to the foliation $\mathcal{V}$. It is clear that, for every $Z\in\operatorname{Poiss}_{\mathcal{V}}(M,P)$, the restriction to a given leaf $L$ of $\mathcal{V}$ is a Poisson vector field of $P_{L}$, $Z|_{L}\in\operatorname{Poiss}(L,P_{L})$. Note also that the morphism $P^{\sharp}:\Gamma(\mathbb{V}^{\mathbb{*}})\rightarrow\Gamma(\mathbb{V})$ associated to the leaf-tangent Poisson structure $P$ induces a linear mapping in cohomology $(P^{\sharp})^{*}:H_{\operatorname{dR}}^{1}(\mathcal{V}) \rightarrow \frac{\operatorname{Poiss}_{\mathcal{V}}(M,P)}{\operatorname{Ham}(M,P)}\subset H_{P}^{1}(M)$ by $(P^{\sharp})^{*}[\mu]:=[P^{\sharp}\mu]$. We call the quotient $\frac{\operatorname{Poiss}_{\mathcal{V}}(M,P)}{\operatorname{Ham}(M,P)}$ the \emph{first cohomology of the Poisson foliation} $(M,\mathcal{V},P)$, which is just the first cohomology of the Lie algebroid $(\mathbb{V}^{*},\iota\circ P^{\sharp},\{,\}_{P})$. Here, $\iota:\mathbb{V}\hookrightarrow TM$ is the inclusion map, and $\{,\}_{P}$ denotes the bracket of foliated 1-forms induced by $P$ \cite{KS-00}.

Suppose that $\mathcal{V}$ is orientable and fix a leaf-wise volume form $\tau\in\Gamma(\wedge^{\operatorname{top}}\mathbb{V}^{*})$. The \emph{modular vector field of the Poisson foliation} $(M,\mathcal{V},P)$ with respect to $\tau$ is the leaf-tangent vector field $Z_{P}^{\tau}\in\Gamma(\mathbb{V})$ defined by the equality
\begin{equation}\label{eq:DefFolMod}
\mathbf{i}_{Z_{P}^{\tau}}\tau=-\operatorname{d}_{\mathcal{V}}\mathbf{i}_{P}\tau.
\end{equation}
It follows from \eqref{eq:DefFolMod} that $Z_{P}^{\tau}\in\operatorname{Poiss}(M,P)$ and $Z_{P}^{f\tau}=Z_{P}^{\tau}-P^{\sharp}\operatorname{d}f$ for all nowhere vanishing $f\in C^{\infty}(M)$. Hence, there is a well-defined cohomology class of the Poisson foliation $(M,\mathcal{V},P)$
\begin{equation*}
\operatorname{Mod}(M,\mathcal{V},P):=[Z_{P}^{\tau}]\in\frac{\operatorname{Poiss}_{\mathcal{V}}(M,P)}{\operatorname{Ham}(M,P)}
\end{equation*}
which can be called the \emph{modular class} of the Poisson foliation $(M,\mathcal{V},P)$, or shortly, the \emph{foliated modular class}.

It is clear that in the case when a Poisson foliation $\mathcal{V}=\{M\}$ consists of a single leaf, the foliated modular class of $(M,\mathcal{V},P)$ just coincides with the modular class of the Poisson manifold $(M,P)$.

Note that the modular class of a Poisson foliation $(M,\mathcal{V},P)$ can be viewed as a particular case of the more general notion of the modular class of the corresponding triangular Lie bi-algebroid $(\mathbb{V},P)$, $P\in\Gamma(\wedge^{2}\mathbb{V})$ \cite{KS-00,KS-08}.

The Poisson foliation $(M,\mathcal{V},P)$ is said to be \emph{unimodular} if $\operatorname{Mod}(M,\mathcal{V},P)=0$. Since the foliated differential $\operatorname{d}_{\mathcal{V}}$ and the leaf-wise volume form $\tau$ restrict to the exterior differential and a volume form $\tau_{L}=\iota_{L}^{*}\tau$ on each leaf $L$ of $\mathcal{V}$, we conclude from \eqref{eq:DefFolMod} that the restriction of the modular vector field $Z_{P}^{\tau}$ to the leaf $L$ is the modular vector field with respect to $\tau_{L}$ of the Poisson structure $P_{L}$, $Z_{P}^{\tau}|_{L}=Z_{P_{L}}^{\tau_{L}}$. Therefore, the unimodularity of the Poisson foliation implies the unimodularity of each leaf. But the converse is not necessarily true.

Here are some useful properties of the modular vector field of the Poisson foliation $(M,\mathcal{V},P)$. Note that, for all $X\in\mathfrak{X}_{\mathrm{pr}}(M,\mathcal{V})$, we have $[X,\Gamma(\wedge^{\bullet}\mathbb{V})]\subset\Gamma(\wedge^{\bullet}\mathbb{V})$, where $[\cdot,\cdot]$ denotes the Schouten-Nijenhuis bracket for multivector fields \cite{DZ}. By definition \eqref{eq:DefFolMod}, and from the standard commuting relations between the operators $\operatorname{L}_{X}$, $\operatorname{d}_{\mathcal{V}}$, and $\mathbf{i}_{A}$,
$A\in\Gamma(\wedge^{\bullet}\mathbb{V})$, we derive the following properties of the modular vector field $Z_{P}^{\tau}$
\begin{align}
\operatorname{L}_{Z_{P}^{\tau}}f&=\operatorname{div}^{\tau}(P^{\sharp}\operatorname{d}f),\label{eq:DivHam}\\
[Z_{P}^{\tau},X]&=P^{\sharp}\operatorname{d}(\operatorname{div}^{\tau}(X)),\label{eq:ModAndPoiss}
\end{align}
for any $f\in C^{\infty}(M)$ and $X\in\mathfrak{X}_{\mathrm{pr}}(M,\mathcal{V})\cap\operatorname{Poiss}(M,P)$.

Furthermore, given a connection form $\gamma\in\Gamma(T^{*}M\otimes\mathbb{V})$ on $(M,\mathcal{V})$, the modular vector field of $(M,\mathcal{V},P)$ relative to $\tau\in\Gamma(\wedge^{\operatorname{top}}\mathbb{V}^{*})$ is determined by
\begin{equation}
-\mathbf{i}_{Z_{P}^{\tau}}\tau_{\gamma}=\operatorname{d}_{0,1}^{\gamma}\mathbf{i}_{P}\tau_{\gamma}.\label{eq:FolModConnect}
\end{equation}

If, in addition to the orientability of $\mathcal{V}$, the manifold $M$ is orientable (or, equivalently, $\mathcal{V}$ is transversally orientable), then we have a relation between the modular class of the Poisson structure $P$ on $M$, the modular class of the Poisson foliation $(M,\mathcal{V},P)$ and the Reeb class of $\mathcal{V}$.

\begin{proposition}\label{prop:ModsAndReeb}
Let $(M,\mathcal{V},P)$ be an orientable and transversally orientable Poisson foliation. Then, the modular class $\operatorname{Mod}(M,P)$ of the Poisson manifold $P\in\Gamma(\wedge^{2}\mathbb{V})$ is related to the foliated modular class $\operatorname{Mod}(M,\mathcal{V},P)$ and the Reeb class $\operatorname{Mod}(\mathcal{V})$ of $\mathcal{V}$ by the formula
\begin{equation}\label{eq:ModsAndReeb}
\operatorname{Mod}(M,P)=\operatorname{Mod}(M,\mathcal{V},P)-(P^{\sharp})^{*}\operatorname{Mod}(\mathcal{V}).
\end{equation}
\end{proposition}

\begin{proof}
Let $\tau\in\Gamma(\wedge^{\operatorname{top}}\mathbb{V}^{*})$ be a leaf-wise volume form and $\varsigma\in\Gamma(\wedge^{s}\mathbb{V}^{0})$ a transversal volume element of $\mathcal{V}$, $s=\operatorname{codim}\mathcal{V}$. Pick a connection form $\gamma$ on $(M,\mathcal{V})$ associated to a normal bundle $\mathbb{H}$ of $\mathcal{V}$. Then, $\Omega:=\varsigma\wedge\tau_{\gamma}$ is a volume form on $M$. Let $Z_{P}^{\Omega}$ and $Z_{P}^{\tau}$ be the modular vector fields of $(M,P)$ and $(M,\mathcal{V},P)$ with respect to the volume forms $\Omega$ and $\tau$, respectively. Consider also the 1-form $\lambda_{\varsigma}^{\gamma}\in\Gamma(\mathbb{H}^{0})$ given by \eqref{R3}. We claim that
\begin{equation}\label{eq:ModsAndReeb2}
Z_{P}^{\Omega}=Z_{P}^{\tau}-P^{\sharp}\lambda_{\varsigma}^{\gamma}.
\end{equation}
Indeed, by bigrading arguments and equality \eqref{eq:FolModConnect}, we get
\begin{align*}
-\mathbf{i}_{Z_{P}^{\Omega}}\Omega &= \operatorname{d}\mathbf{i}_{P}\Omega = \operatorname{d}_{0,1}^{\gamma}(\varsigma\wedge\mathbf{i}_{P}\tau_{\gamma}) = \operatorname{d}_{0,1}^{\gamma}\varsigma\wedge\mathbf{i}_{P}\tau_{\gamma} + (-1)^{s}\varsigma\wedge \operatorname{d}_{0,1}^{\gamma}\mathbf{i}_{P}\tau_{\gamma}\\
&= \lambda_{\varsigma}^{\gamma}\wedge\varsigma\wedge\mathbf{i}_{P}
\tau_{\gamma}-(-1)^{s}\varsigma\wedge\mathbf{i}_{Z_{P}^{\tau}}\tau_{\gamma} = (-1)^{s}\varsigma\wedge\mathbf{i}_{\mathbf{i}_{\lambda_{\varsigma}^{\gamma}}P}\tau_{\gamma} - (-1)^{s}\varsigma\wedge\mathbf{i}_{Z_{P}^{\tau}}\tau_{\gamma}\\
&=\mathbf{i}_{P^{\sharp}\lambda_{\varsigma}^{\gamma}}\Omega-\mathbf{i}_{Z_{P}^{\tau}}\Omega.
\end{align*}
Here we have applied, on the second and fifth steps, the identity
\begin{equation}\label{eq:Formula}
\mathbf{i}_{A}(\alpha\wedge\beta)=(-1)^{|A|}(\alpha\wedge\mathbf{i}_{A}\beta-\mathbf{i}_{\mathbf{i}_{\alpha}A}\beta),
\end{equation}
valid for all $\alpha\in\Gamma(T^{*}M),\beta\in\Gamma(\wedge^{\bullet}T^{*}M)$ and $A\in\Gamma(\wedge^{\bullet}TM)$. Thus, we have proved \eqref{eq:ModsAndReeb2}, which implies \eqref{eq:ModsAndReeb}.
\end{proof}

The following corollary to Proposition \ref{prop:ModsAndReeb} gives us a unimodularity criterion for a class of Poisson foliations coming from fibrations.

\begin{corollary}\label{cor:LocallyTrivialMod}
Let $(M\overset{\pi}{\rightarrow}S,P)$ be a locally trivial Poisson fiber bundle. Suppose that the total space $M$ and the
base $S$ are orientable. If the typical fiber $F$ is a unimodular Poisson manifold, then $\operatorname{Mod}(M,P)=0$.
\end{corollary}

\begin{proof}
Consider the regular foliation $\mathcal{V}:=\{\pi^{-1}(x)\}_{x\in S}$ on $M$ associated to the projection $\pi$. The orientability of the base implies $\operatorname{Mod}(\mathcal{V})=0$ (see, Example \ref{ex:FiberReeb}). Then, by \eqref{eq:ModsAndReeb}, it suffices to show that $\operatorname{Mod}(M,\mathcal{V},P)=0.$ Fix a nowhere vanishing top section $\tau\in\Gamma(\wedge^{\operatorname{top}}\mathbb{V}^{*})$, where $\mathbb{V}:=\ker d\pi$ is the vertical bundle, and a family of trivializations $M_{i}:=\pi^{-1}(U_{i})\cong U_{i}\times F$ over open sets $U_{i}$ which cover $S$. By the unimodularity hypothesis for $F$, one can equip each trivial Poisson bundle $(\pi_{i}:M_{i}\rightarrow U_{i},P_{i}:=P|_{M_{i}})$ with a leaf-wise volume form of positive orientation $\tau_{i}\in\Gamma(\wedge^{\operatorname{top}}\mathbb{V}^{*}|_{M_{i}})$ such that the corresponding modular vector field is zero, $Z_{P_{i}}^{\tau_{i}}$ $=0$. From here and the partition of unity argument, we conclude that there exists a global leaf-wise volume form $\tau_{0}$ of $\mathcal{V}$ such that $\operatorname{L}_{P^{\sharp}\operatorname{d}f}\tau_{0}=0$ for all $f\in C^{\infty}(M)$.
\end{proof}

In the regular case, as a consequence of Proposition \ref{prop:ModsAndReeb}, we recover the result due to \cite{We-97,AB-03} which says that the modular class of an orientable regular Poisson manifold is determined by the Reeb class of its symplectic foliation. Indeed, suppose that the Poisson manifold $(M,P)$ is regular with $\operatorname{rank}P=2s$. Let $\mathcal{V}=\mathcal{S}$ be the symplectic foliation of $P$ equipped with the leaf-wise symplectic form $\omega$. Then, the canonical leaf-wise volume form $\tau=\wedge^{s}\omega$ of the symplectic foliation is such that the modular vector field of the Poisson foliation $(M,\mathcal{S},P)$ is zero, $Z_{P}^{\tau}=0$. If, in addition, $M$ is orientable, then the symplectic foliation is transversally orientable. Therefore, in this case, formula \eqref{eq:ModsAndReeb} reads $\operatorname{Mod}(M,P)=-(P^{\sharp})^{*}\operatorname{Mod}(\mathcal{S})$.

\section{Modular Vector Fields of Coupling Poisson Structures}\label{sec:CoupMod}

Let $\mathcal{V}$ be a regular foliation of the smooth manifold $M$. Consider the tangent bundle $\mathbb{V}=T\mathcal{V}$ and its annihilator $\mathbb{V}^{0}\subset T^{*}M$.

Suppose we are given a $\mathcal{V}$-\emph{coupling Poisson structure} \cite{Vo-01,Va-04} $\Pi\in\Gamma(\wedge^{2}TM)$, that is, a Poisson bivector on $M$ such that
\begin{equation}\label{NB}
\mathbb{H}:=\Pi^{\sharp}(\mathbb{V}^{0})
\end{equation}
is a normal bundle of the foliation,
\begin{equation}\label{eq:SplitHV}
TM=\mathbb{H}\oplus\mathbb{V}.
\end{equation}
Then, the bigraded decomposition of $\Pi$ with respect to \eqref{eq:SplitHV} is of the form $\Pi=\Pi_{2,0}+\Pi_{0,2}$, where $\Pi_{H}:=\Pi_{2,0}\in\Gamma(\wedge^{2}\mathbb{H})$ is a bivector field of constant rank, with $\operatorname{rank}\Pi_{H}=\operatorname{rank}\mathbb{H}$, and $\Pi_{0,2}\in\Gamma(\wedge^{2}\mathbb{V})$ is a Poisson tensor on $M$ tangent to the foliation $\mathcal{V}$. Therefore, we can associate to the $\mathcal{V}$-coupling Poisson structure $\Pi$ the Poisson foliation $(M,\mathcal{V},P:=\Pi_{0,2})$. Notice that the characteristic distribution of $\Pi$ splits as $\Pi^{\sharp}(T^{*}M)=\mathbb{H}\oplus P^{\sharp}(\mathbb{V}^{*})$. Hence, $\operatorname{rank}\Pi=\operatorname{rank}\mathbb{H}+\operatorname{rank}P$, so the set of singular points of $\Pi$ and $P$ coincide.

Moreover, the restriction $\Pi_{H}^{\sharp}|_{\mathbb{V}^{0}}:\mathbb{V}^{0}\rightarrow\mathbb{H}$ is a vector bundle isomorphism and hence one can define an $\mathbb{H}$-nondegenerated 2-form $\sigma\in\Gamma(\wedge^{2}\mathbb{V}^{0})$, called the \emph{coupling form}, by
\begin{equation}\label{eq:CouplingForm}
\sigma^{\flat}|_{\mathbb{H}}:=-(\Pi_{H}^{\sharp}|_{\mathbb{V}^{0}})^{-1}:\mathbb{H}\rightarrow\mathbb{V}^{0}.
\end{equation}
Let $\gamma$ be the connection form on $(M,\mathcal{V})$ associated to the normal bundle $\mathbb{H}$ in \eqref{NB}. Then, the \emph{geometric data} $(\gamma,\sigma,P)$ associated to the coupling Poisson tensor $\Pi$ satisfy the following structure equations \cite{Vo-01,Va-04,Vo-05}
\begin{align}
[X,P]=0 & \qquad\forall X\in\Gamma_{\mathrm{pr}}(\mathbb{H}),\label{eq:SE0}\\
R^{\gamma}(X,Y)=-P^{\sharp}\operatorname{d}[\sigma(X,Y)] & \qquad\forall X,Y\in\Gamma_{\mathrm{pr}}(\mathbb{H}),\label{eq:SE1}\\
\operatorname{d}_{1,0}^{\gamma}\sigma=0 & . \label{eq:SE2}
\end{align}

In particular, the first equation means that $\gamma$ is a \emph{Poisson connection} on $(M,\mathcal{V},P)$. Moreover, by the $\mathbb{H}$-nondegeneracy property of the coupling form $\sigma$, the foliation $\mathcal{V}$ admits a canonical transversal volume element given by $l$ times the product of $\sigma$, $\sigma^{l} := \sigma\wedge\ldots\wedge\sigma \in \Gamma(\wedge^{\operatorname{top}}\mathbb{V}^{0})$, where $2l:=\operatorname{rank}\mathbb{V}^{0}=\operatorname{rank}\mathbb{H}$.

Now, assume that $\mathcal{V}$ is an orientable foliation. Then, one can associate to each leaf-wise volume form $\tau\in\Gamma(\wedge^{\operatorname{top}}\mathbb{V}^{*})$ of $\mathcal{V}$ a volume form $\Omega$ of $M$ by $\Omega:=\sigma^{l}\wedge\tau_{\gamma}$. Moreover, recall that $\tau$ gives rise to the divergence 1-form $\theta_{\tau}^{\gamma}\in\Gamma(\mathbb{V}^{0})$, defined by \eqref{eq:DefTheta}, and the modular vector field $Z_{P}^{\tau}$ of the Poisson foliation $(M,\mathcal{V},P)$, introduced in \eqref{eq:DefFolMod}. We describe the modular vector fields of coupling Poisson structures in terms of these objects.

\begin{proposition}\label{prop:CoupMod}
Let $\Pi$ be a coupling Poisson structure on the orientable foliated manifold $(M,\mathcal{V})$ associated to geometric data $(\gamma,\sigma,P)$. Fix a leaf-wise volume form $\tau$ of $\mathcal{V}$ and consider the volume form $\Omega:=\sigma^{l}\wedge\tau_{\gamma}$ of $M$. If $Z:=Z^{\Omega}_{\Pi}$ is the corresponding modular vector field, then the bigraded components of $Z$ relative to the splitting \eqref{eq:SplitHV} are given by
\begin{equation}\label{eq:CoupMod}
Z_{1,0}=-\Pi^{\sharp}(\theta_{\tau}^{\gamma}),\qquad Z_{0,1}=Z_{P}^{\tau}.
\end{equation}
\end{proposition}

\begin{proof}
By the definition of the modular vector field $Z$ and using the bigraded decompositions of $\operatorname{d}$ and $\Pi$, we have
\begin{align}\label{Re}
-\mathbf{i}_{Z}\Omega &= \operatorname{d}\mathbf{i}_{\Pi}\Omega = \operatorname{d}_{1,0}^{\gamma}\mathbf{i}_{\Pi_{H}}\Omega + \operatorname{d}_{0,1}^{\gamma}\mathbf{i}_{P}\Omega + \operatorname{d}_{2,-1}^{\gamma}\mathbf{i}_{\Pi_{H}}\Omega\nonumber\\
&= \operatorname{d}_{1,0}^{\gamma}\mathbf{i}_{\Pi_{H}}\sigma^{l}\wedge\tau_{\gamma} + \mathbf{i}_{\Pi_{H}}\sigma^{l}\wedge \operatorname{d}_{1,0}^{\gamma}\tau_{\gamma} + \operatorname{d}_{0,1}^{\gamma}\sigma^{l}\wedge\mathbf{i}_{P}\tau_{\gamma}\nonumber\\
&\hspace{7.5em} + \sigma^{l}\wedge\operatorname{d}_{0,1}^{\gamma}\mathbf{i}_{P}\tau_{\gamma} - \mathbf{i}_{\mathbf{i}_{\Pi_{H}}R^{\gamma}}\Omega.
\end{align}
It follows from \eqref{eq:CouplingForm} that $\mathbf{i}_{\Pi_{H}}\sigma^{l}=-l\sigma^{l-1}$. This together with \eqref{eq:SE2} implies $\operatorname{d}_{1,0}^{\gamma}\mathbf{i}_{\Pi_{H}}\sigma^{l}=0$. On the other hand, there exists a 1-form $\Lambda\in\Gamma(\mathbb{H}^{0})$ satisfying the relation $\operatorname{d}_{0,1}^{\gamma}\sigma^{l}=\Lambda\wedge\sigma^{l}$. Then, from \eqref{Re}, by using \eqref{eq:FolModConnect}, \eqref{eq:DefTheta} and \eqref{eq:Formula}, we get
\begin{align*}
-\mathbf{i}_{Z}\Omega &= \theta_{\tau}^{\gamma}\wedge\mathbf{i}_{\Pi_{H}}\sigma^{l}\wedge\tau_{\gamma} + \Lambda\wedge\sigma^{l}\wedge\mathbf{i}_{P}\tau_{\gamma} - \sigma^{l}\wedge\mathbf{i}_{Z_{P}^{\tau}}\tau_{\gamma} - \mathbf{i}_{\mathbf{i}_{\Pi_{H}}R^{\gamma}}\Omega\\
&= \mathbf{i}_{\Pi_{H}^{\sharp}(\theta_{\tau}^{\gamma})}\Omega + \mathbf{i}_{P^{\sharp}\Lambda}\Omega - \mathbf{i}_{Z_{P}^{\tau}}\Omega - \mathbf{i}_{\mathbf{i}_{\Pi_{H}}R^{\gamma}}\Omega.
\end{align*}
It is left to show
\begin{equation}\label{SH}
\mathbf{i}_{\Pi_{H}}R^{\gamma}=P^{\sharp}\Lambda.
\end{equation}
Consider the $2l$-vector field given by $l$ times the product of $\Pi_{H}$, $\Pi_{H}^{l} :=\Pi_{H} \wedge\ldots\wedge\Pi_{H}$. Using again identities \eqref{eq:Formula}, \eqref{eq:SE2} and the bigrading argument, we evaluate
\begin{align*}
(\mathbf{i}_{\Pi_{H}^{l}}\sigma^{l})\Lambda &= \mathbf{i}_{\Pi_{H}^{l}}(\Lambda\wedge\sigma^{l}) = \mathbf{i}_{\Pi_{H}^{l}}(\operatorname{d}\sigma^{l}) = \mathbf{i}_{\Pi_{H}^{l}}(l\operatorname{d}\sigma\wedge\sigma^{l-1})\\
&= -\mathbf{i}_{\Pi_{H}^{l}}(\operatorname{d}\sigma\wedge\mathbf{i}_{\Pi_{H}}\sigma^{l}) = -(\mathbf{i}_{\Pi_{H}^{l}}\sigma^{l})\mathbf{i}_{\Pi_{H}}\operatorname{d}\sigma.
\end{align*}
From here and taking into account that $\mathbf{i}_{\Pi_{H}^{l}}\sigma^{l}\neq0$, we conclude $\Lambda=-\mathbf{i}_{\Pi_{H}}\operatorname{d}\sigma$. On the other hand, the curvature identity \eqref{eq:SE1} implies the equality $\mathbf{i}_{\Pi_{H}}R^{\gamma} = -P^{\sharp}\mathbf{i}_{\Pi_{H}}\operatorname{d}\sigma$ which together with the above representation for $\Lambda$ proves \eqref{SH}.
\end{proof}

As mentioned above, the set of singular points of the coupling Poisson structure $\Pi$ coincides with the set of singular points of its leaf-tangent part $P=\Pi_{0,2}$. From the relations \eqref{eq:CoupMod}, we derive the following information on the behavior of the modular vector fields of $\Pi$ at the singular points.

\begin{corollary}
A modular vector field of the Poisson manifold $(M,\Pi)$ is tangent to the symplectic foliation of $\Pi$ at a point $x\in M$ if and only if a modular vector field $Z_{P}^{\tau}\in\Gamma(\mathbb{V})$ of the Poisson foliation $(M,\mathcal{V},P)$ is tangent to the symplectic foliation of $P$ at $x$. In particular, this is true if $x$ is a regular point of $P$.
\end{corollary}

\begin{remark}
More generally, for a Poisson submanifold $N$ of a Poisson manifold $(M,\Pi)$, one can introduce the notion of a relative modular class of $N$ \cite{CaFe-13}. If this class vanishes, then the modular vector field of $(M,\Pi)$ is tangent to $N$. In particular, this criterion can be applied when $N$ is a symplectic leaf.
\end{remark}

Notice that the 1-form $\Lambda\in\Gamma(\mathbb{H}^{0})$, arising in \eqref{SH}, just coincides with the 1-form $\lambda_{\varsigma}^{\gamma}$ defined by \eqref{R3} for $\varsigma=\sigma^{l}$, whose $\operatorname{d}_{0,1}^{\gamma}$-cohomology class gives the Reeb class of the foliation $\mathcal{V}$. Moreover, by the curvature relation \eqref{SH} and Proposition \ref{prop:ModsAndReeb}, we conclude that if $\gamma$ is flat, $R^{\gamma}=0$, then
\begin{equation}\label{eq:ModsFlat}
\operatorname{Mod}(M,P)=\operatorname{Mod}(M,\mathcal{V},P).
\end{equation}

Now, let us consider the Lie algebra $\operatorname{Poiss}_{\mathcal{V}}(M,P)$ of all $\mathcal{V}$-tangent Poisson vector fields of $P$. Then, the projection $\gamma:TM\rightarrow\mathbb{V}$ along $\mathbb{H}$ in decomposition \eqref{eq:SplitHV} induces the linear mapping \cite{VeVo-16}
\begin{equation}\label{eq:CohomProy}
\gamma^{*}:H_{\Pi}^{1}(M)\rightarrow\frac{\operatorname{Poiss}_{\mathcal{V}}(M,P)}{\operatorname{Ham}(M,P)}\subset H_{P}^{1}(M)
\end{equation}
from the first Poisson cohomology of $(M,\Pi)$ to the first cohomology of the Poisson foliation $(M,\mathcal{V},P)$. As a consequence of Proposition \ref{prop:CoupMod}, this map is natural with respect to the modular classes.

\begin{corollary}\label{cor:ProyMod}
The quotient map \eqref{eq:CohomProy} takes the modular class of the Poisson manifold $(M,\Pi)$ to the modular class of the Poisson foliation $(M,\mathcal{V},P)$, $\gamma^{*}(\operatorname{Mod}(M,\Pi)) = \operatorname{Mod}(M,\mathcal{V},P)$.
\end{corollary}

\section{Gauge Transformations}\label{sec:gauge}

As we already mentioned above, according to \cite{We-97,AB-03}, the unimodularity of an orientable regular Poisson manifold $(M,\Pi)$ is equivalent to the triviality of the Reeb class of the characteristic (symplectic) foliation $\mathcal{S}$ of $\Pi$. In other words, this means that the unimodularity property is independent of the leaf-wise symplectic structure on $\mathcal{S}$ in the following sense: if $\widetilde{\Pi}$ is another regular Poisson structure on $M$ which has the same characteristic foliation $\mathcal{S}$, then the unimodularity of $\Pi$ implies $\operatorname{Mod}(M,\widetilde{\Pi})=0$. But this fact is no longer true in the singular case.

For example, let us consider on $\mathbb{R}^{3}$, with coordinate functions $(x_{1},x_{2},x_{3})$, the linear Poisson structure $\Pi=\tfrac{1}{2}\epsilon_{ijk}x_{i}\frac{\partial}{\partial x_{j}}\wedge\frac{\partial}{\partial x_{k}}$ associated to the Lie algebra $\mathfrak{so}(3)$. Here $\epsilon_{ijk}$ are the \emph{Levi-Civita symbols}. We are using the Einstein summation convention. Consider also the homogeneous Poisson structure $\widetilde{\Pi} = f\Pi$, where $f(x_{1},x_{2},x_{3}):=x_{1}^{4}+x_{2}^{4}+x_{3}^{4}$. It is clear that the characteristic foliations of these structures coincide. Computing the corresponding modular vector fields with respect to the Euclidean volume form $\Omega$ in $\mathbb{R}^{3}$, we get $Z_{\Pi}^{\Omega}\equiv0$ and
\[
Z_{\widetilde{\Pi}}^{\Omega} = 2\epsilon_{ijk}\left(x_{i}^{3}x_{j}-x_{i}x_{j}^{3}\right)\frac{\partial}{\partial x_{k}}\neq0.
\]
This shows that $\Pi$ is unimodular, while $\widetilde{\Pi}$ is not, even though they have the same characteristic foliation.

On the other hand, there exists an equivalence relation for (possibly singular) Poisson structures, called the \emph{gauge equivalence} \cite{SeWe-01}, which preserves the unimodularity property.

Let $(M,\Pi)$ be a Poisson manifold. Suppose we are given a closed 2-form $B$ on $M$ such that the endomorphism
\begin{equation}\label{eq:Gau1}
(\operatorname{Id}-B^{\flat}\circ\Pi^{\sharp}):T^{*}M\rightarrow T^{*}M
\end{equation}
is invertible. Then, there exists a Poisson bivector field $\widetilde{\Pi}$ on $M$ defined by the relation $\widetilde{\Pi}^{\sharp}=\Pi^{\sharp}\circ(\operatorname{Id}-B^{\flat}\circ\Pi^{\sharp})^{-1}$ and represents the result of $\Pi$ under the gauge transformation induced by $B$ \cite{SeWe-01,Br-05}. In this case, we say that $\widetilde{\Pi}$ is \emph{gauge equivalent} to $\Pi$. The gauge transformation modifies only the leaf-wise symplectic form of $\Pi$ by means of the pull-back of the closed 2-form $B$, preserving the characteristic foliation. Furthermore, gauge transformations preserve the unimodularity property.

\begin{proposition}\label{prop:Gauge}
If $\Pi$ and $\widetilde{\Pi}$ are gauge equivalent Poisson structures on $M$, then
\begin{equation}\label{eq:ModGauge}
\operatorname{Mod}(M,\Pi)=0~~\Longleftrightarrow~~\operatorname{Mod}(M,\widetilde{\Pi})=0.
\end{equation}
\end{proposition}

\begin{proof}
The modular class $\operatorname{Mod}(M,\Pi)$ of the orientable Poisson manifold $(M,\Pi)$ is one-half the modular class of the cotangent bundle $T^{*}M$ of $M$ with the Lie algebroid structure defined by $\Pi$ \cite{ELW}. As is known \cite{SeWe-01}, the map \eqref{eq:Gau1} induced by the gauge transformation is an isomorphism between the cotangent Lie algebroids associated to $\Pi$ and $\widetilde{\Pi}$. This proves the statement.
\end{proof}

\section{Unimodularity Criteria}\label{sec:Unimod}

Assume that on the orientable foliated manifold $(M,\mathcal{V})$, we are given a $\mathcal{V}$-coupling Poisson structure $\Pi$ associated to geometric data $(\gamma,\sigma,P)$. Our point is to formulate some conditions for the unimodularity of $\Pi$ in terms of the geometric data.

The following fact is a direct consequence of Corollary \ref{cor:ProyMod}.

\begin{lemma}\label{lemma:Unimods}
The unimodularity of the coupling Poisson structure $\Pi$ implies the unimodularity of the Poisson foliation $(M,\mathcal{V},P)$.
\end{lemma}

Therefore, a necessary condition for vanishing of the modular class of $\Pi$ is $\operatorname{Mod}(M,\mathcal{V},P)=0$. Moreover, it follows from Proposition \ref{prop:ModsAndReeb} that the unimodularity of $\Pi$ implies the unimodularity of the leaf-tangent Poisson structure $P$ in the case when the Reeb class of the foliation $\mathcal{V}$ is trivial.

The next criterion follows from Proposition \ref{prop:CoupMod} and the following well-known fact \cite{We-97}: a Poisson manifold is unimodular if and only if the modular vector field is zero with respect to a certain volume form.

\begin{lemma}\label{lemma:UnimodTau}
The $\mathcal{V}$-coupling Poisson structure $\Pi$ is unimodular if and only if there exists a leaf-wise volume form $\tau\in\Gamma(\wedge^{\operatorname{top}}\mathbb{V}^{*})$, $\mathbb{V}=T\mathcal{V}$, such that
\[
Z_{P}^{\tau}=0\qquad\text{and}\qquad\operatorname{d}_{1,0}^{\gamma}\tau_{\gamma}=0.
\]
\end{lemma}

It follows that the unimodularity of $\Pi$ is independent of the coupling form $\sigma$. In other words, the mapping
\[
(\gamma,\sigma,P)\mapsto(\gamma,\widetilde{\sigma},P)
\]
is a foliation-preserving transformation which do not alter the unimodularity property, provided that $\widetilde{\sigma}$ satisfies the nondegeneracy condition and the structure equations \eqref{eq:SE1}, \eqref{eq:SE2}. This is also a ``singular'' analog of the fact that, for a regular Poisson manifold, the unimodularity is independent of the leaf-wise symplectic form.

Now let us describe a special class of gauge transformations which preserve the coupling Poisson structures and naturally appear in the context of the averaging method \cite{VallVo-14}. Consider the case when the gauge form $B$ is exact with a primitive $\mu$ vanishing along the leaves of the foliation $\mathcal{V}$:
\begin{equation}\label{eq:Gau2}
B=-\operatorname{d}\mu,\quad\mu\in\Gamma(\mathbb{V}^{0}).
\end{equation}
Then, assuming that the map \eqref{eq:Gau1} is invertible, one can show \cite{VallVo-14} that the Poisson structure $\widetilde{\Pi}$ resulting of the gauge transformation of $\Pi$ is again $\mathcal{V}$-coupling. Furthermore, if $(\widetilde{\gamma},\widetilde{\sigma},\widetilde{P})$ is the geometric data associated to $\widetilde{\Pi}$, then $\widetilde{P}=P$ and $\widetilde{\gamma}$ is related to $\gamma$ by
\begin{equation}\label{eq:Gau3}
\gamma(X)-\widetilde{\gamma}(X)=P^{\sharp}\operatorname{d}[\mu(X)] \qquad\forall X\in\mathfrak{X}_{\mathrm{pr}}(M).
\end{equation}

Fix a nowhere vanishing section $\tau\in\Gamma(\wedge^{\operatorname{top}}\mathbb{V}^{\ast})$ and let us look at the corresponding divergence forms $\theta_{\tau}^{\widetilde{\gamma}}$ and $\theta_{\tau}^{\gamma}$. By relations \eqref{eq:Theta1} and \eqref{eq:Gau3}, for every $X\in\mathfrak{X}_{\operatorname{pr}}(M,\mathcal{V}),$ we have
\begin{equation}\label{eq:Gau4}
\theta_{\tau}^{\widetilde{\gamma}}(X)-\theta_{\tau}^{\gamma}(X) = \operatorname{div}^{\tau}(P^{\sharp}\operatorname{d}\mu(X)) = \operatorname{L}_{Z_{P}^{\tau}}(\mu(X)).
\end{equation}
Here, we used the identity \eqref{eq:DivHam}. Formulas \eqref{eq:CoupMod}, \eqref{eq:Gau4}, give the transition rule for the modular vector fields of $\Pi$ and $\widetilde{\Pi}$.

Next, if $\Pi$ is unimodular, then by Lemma \ref{lemma:UnimodTau} we can choose a leaf-wise volume form $\tau$ of $\mathcal{V}$ such that $Z_{P}^{\tau}=0$ and $\theta_{\tau}^{\gamma}=0$. In this case, we have
$\theta_{\tau}^{\tilde{\gamma}}=\theta_{\tau}^{\gamma}=0$. Hence by Proposition \ref{prop:CoupMod}, if the modular vector field of $\Pi$ with respect to the volume form $\Omega=\sigma^{l}\wedge\tau_{\gamma}$ is zero, then the modular vector field of $\widetilde{\Pi}$ with respect to $\widetilde{\Omega}=\sigma^{l}\wedge\tau_{\tilde{\gamma}}$ is also zero.

\paragraph{Cohomological Obstructions to the Unimodularity.} By Lemma \ref{lemma:Unimods}, a necessary condition for the unimodularity of the $\mathcal{V}$-coupling Poisson structure $\Pi$ on $(M,\mathcal{V})$ is the unimodularity of the Poisson foliation $(M,\mathcal{V},P)$. We will show that this condition is not sufficient, since there exists a cohomological obstruction to the unimodularity of $\Pi$.

Consider the Poisson foliation $(M,\mathcal{V},P,\gamma)$ equipped with the Poisson connection $\gamma$ corresponding to the normal bundle $\mathbb{H}$ in \eqref{NB}. Then, one can associate to this setup the following cochain complex $(\mathcal{C}^{\bullet},\overline{\operatorname{d}}^{\gamma})$, where the subspaces $\mathcal{C}^{p}\subset\Gamma(\wedge^{p}\mathbb{V}^{0})$ are defined by
\begin{equation}\label{eq:dRCasim}
\mathcal{C}^{p}:=\{\beta\in\Gamma(\wedge^{p}\mathbb{V}^{0})\mid\mathbf{i}_{X_{1}}\ldots\mathbf{i}_{X_{p}}\beta\in\operatorname{Casim}(M,P),~\forall
~X_{i}\in\mathfrak{X}_{\mathrm{pr}}(M,\mathcal{V})\}
\end{equation}
and $\overline{\operatorname{d}}^{\gamma}:=\operatorname{d}_{1,0}^{\gamma}|_{\mathcal{C}^{\bullet}}$ is the restriction of $\operatorname{d}_{1,0}^{\gamma}$ to $\mathcal{C}^{\bullet}$. Therefore, $\mathcal{C}^{p}$ consists of $p$-forms on $M$ vanishing along the leaves of $\mathcal{V}$ and taking values in the space of Casimir functions of $P$ on the projectable vector fields.

There exists the following short exact sequence \cite{VeVo-16}:
\begin{equation}\label{eq:ShortExact}
0\rightarrow H_{\overline{\operatorname{d}}^{\gamma}}^{1}\overset{(\Pi_{H}^{\sharp})^{*}}{\longrightarrow}H_{\Pi}^{1}(M)\overset{\gamma^{*}}{\longrightarrow} \frac{\ker\rho}{\operatorname{Ham}(M,P)}\rightarrow0,
\end{equation}
where $\rho:\mathcal{A}^{\gamma}\rightarrow H_{\overline{\operatorname{d}}^{\gamma}}^{2}$ is a morphism from a Lie subalgebra $\mathcal{A}^{\gamma}\subset\operatorname{Poiss}_{\mathcal{V}}(M,P)$, associated to the pair $(\gamma,P)$, to the second cohomology space of $(\mathcal{C}^{\bullet},\overline{\operatorname{d}}^{\gamma})$.

According to Corollary \ref{cor:ProyMod} and \eqref{eq:ShortExact}, if
\begin{equation}\label{MM}
\operatorname{Mod}(M,\mathcal{V},P)=0,
\end{equation}
then there exists a unique cohomology class in $H_{\overline{\operatorname{d}}^{\gamma}}^{1}$ such that its image under $-(\Pi^{\sharp}_{H})^{*}$ is $\operatorname{Mod}(M,\Pi)$. This cohomology class can be described as follows.

\begin{theorem}\label{teo:Class}
Let $M$ be an orientable manifold, and $\mathcal{V}$ an orientable foliation on $M$. Suppose that the $\mathcal{V}$-coupling Poisson structure $\Pi$ on $M$ satisfies \eqref{MM}. Fix a leaf-wise volume form $\tau\in\Gamma(\wedge^{\operatorname{top}}\mathbb{V}^{*})$ of $\mathcal{V}$ such that $Z_{P}^{\tau}=0$ and consider the Poisson connection $\gamma$ associated to $\mathbb{H}$ in \eqref{NB}. Then, the corresponding divergence form $\theta_{\tau}^{\gamma}$ in \eqref{eq:DefTheta} is a 1-cocycle of the cochain complex $(\mathcal{C}^{\bullet},\overline{\operatorname{d}}^{\gamma})$, $\theta_{\tau}^{\gamma}\in\mathcal{C}^{1}$ and $\overline{\operatorname{d}}^{\gamma}\theta_{\tau}^{\gamma}=0$. Furthermore, the $\overline{\operatorname{d}}^{\gamma}$-cohomology class of $\theta_{\tau}^{\gamma}$ is independent of the choice of $\tau$ and related with the modular class of $\Pi$ by $\operatorname{Mod}(M,\Pi)=-(\Pi_{H}^{\sharp})^{*}[\theta_{\tau}^{\gamma}]$.
\end{theorem}

\begin{proof}
By \eqref{eq:SE0}, every projectable section $X\in\Gamma_{\mathrm{pr}}(\mathbb{H})$ is a Poisson vector field of $P$. Then, by using the condition
$Z_{P}^{\tau}=0$, properties \eqref{eq:Theta1} and \eqref{eq:ModAndPoiss}, we get
\[
P^{\sharp}\operatorname{d}[\theta_{\tau}^{\gamma}(X)] = P^{\sharp}\operatorname{d}[\operatorname{div}^{\tau}(X)] = [Z_{P}^{\tau},X]=0.
\]
Therefore, $\theta_{\tau}^{\gamma}(X)\in\operatorname{Casim}(M,P)$ $\forall X\in\Gamma_{\mathrm{pr}}(\mathbb{H})$ and hence $\theta_{\tau}^{\gamma}\in\mathcal{C}^{1}$. Moreover, relations \eqref{eq:Theta2}, \eqref{eq:SE1} and \eqref{eq:DivHam} imply that $\theta_{\tau}^{\gamma}$ is $\overline{\operatorname{d}}^{\gamma}$-closed. Indeed, for all $X_{1},X_{2}\in\mathfrak{X}_{\mathrm{pr}}(M,\mathcal{V})$,
\begin{align*}
(\overline{\operatorname{d}}^{\gamma}\theta_{\tau}^{\gamma})(X_{1},X_{2}) &= \operatorname{div}^{\tau}(R^{\gamma}(X_{1},X_{2})) = -\operatorname{div}^{\tau}(P^{\sharp}\operatorname{d}[\sigma(X_{1},X_{2})])\\
&= -L_{Z_{P}^{\tau}}[\sigma(X_{1},X_{2})]=0.
\end{align*}
Note that any two leaf-wise volume forms for which the modular vector fields of the Poisson foliation $(M,\mathcal{V},P)$ vanish are related by multiplication of a Casimir function. Thus, it follows from the transition rule \eqref{eq:Theta3} that $[\theta_{\tau}^{\gamma}]\in H_{\overline{\operatorname{d}}^{\gamma}}^{1}$ is independent on the choice of $\tau$. Finally, $\operatorname{Mod}(M,\Pi) = -(\Pi_{H}^{\sharp})^{*}[\theta_{\tau}^{\gamma}]$ follows from \eqref{eq:CoupMod}.
\end{proof}

\begin{corollary}\label{cor:Obstruction}
If the Poisson foliation $(M,\mathcal{V},P)$ associated to the $\mathcal{V}$-coupling Poisson structure $\Pi$ is unimodular, then the unimodularity of $\Pi$ is equivalent to the triviality of the $\overline{\operatorname{d}}^{\gamma}$-cohomology class of $\theta_{\tau}^{\gamma}$, that is, $\operatorname{Mod}(M,\Pi)=0\Longleftrightarrow[\theta_{\tau}^{\gamma}]=0$.
\end{corollary}

\begin{example}
Consider the particular case when the leaf-tangent Poisson structure $P$ is trivial, $P=0$. Then, the coupling Poisson structure $\Pi$ is regular since its characteristic distribution coincides with the normal bundle $\mathbb{H}$. Moreover, $(\mathcal{C}^{\bullet},\overline{\operatorname{d}}^{\gamma})$ identifies with the foliated de Rham complex of the symplectic foliation $\mathcal{S}$ of $\Pi$. In particular, the cohomology class $[\theta_{\tau}^{\gamma}]$ coincides with the Reeb class $\operatorname{Mod}(\mathcal{S})$.
\end{example}

Note that the coupling Poisson structure $\Pi$ with $P=0$ can be characterized as a regular Poisson structure whose symplectic foliation $\mathcal{S}$ admits a transversal foliation $\mathcal{V}$, $TM=T\mathcal{S\oplus}T\mathcal{V}$. So, in this case, the unimodularity criterion of Corollary \ref{cor:Obstruction} recovers the results due to \cite{We-97,AB-03}.

\section{Flat Poisson Foliations}\label{sec:Flat}

Suppose we start with a Poisson foliation $(M,\mathcal{V},P)$ consisting of a regular foliation $\mathcal{V}$ on $M$ and a leaf-tangent Poisson structure $P\in\Gamma(\wedge^{2}\mathbb{V})$. Suppose we are also given a regular foliation $\mathcal{F}$ on $M$ with properties: the tangent bundle
$\mathbb{F}:=T\mathcal{F}$ is complementary to $\mathbb{V}=T\mathcal{V}$, $TM=\mathbb{F}\oplus\mathbb{V}$,
and every $\mathcal{V}$-projectable section $Z$ of $\mathbb{F}$ is a Poisson vector field on $(M,P)$,
\begin{equation}\label{eq:FlatPoissConn}
Z\in\Gamma_{\mathrm{pr}}(\mathbb{F}) ~\Longrightarrow~ L_{Z}P=0.
\end{equation}
In other words, there is a flat Poisson connection $\gamma_{0}\in\Gamma(T^{*}M\otimes\mathbb{V})$ on $(M,\mathcal{V},P)$ associated to the tangent bundle of $\mathcal{F}$, $\mathbb{F}=\ker\gamma_{0}$, and hence $\gamma_{0}:TM\rightarrow\mathbb{V}$ is the projection along $\mathbb{F}$.

Let us associate to the \emph{flat Poisson foliation} $(M,\mathcal{V},P,\mathcal{F})$ the following objects. According to the dual splitting $T^{*}M = \mathbb{V}^{0}\oplus\mathbb{F}^{0}$, we have the bigrading of differential forms on $M$ and the bigraded decomposition of the exterior differential on $M$: $\operatorname{d}=\partial_{\mathcal{F}}+\partial_{\mathcal{V}}$, where $\partial_{\mathcal{F}}:=\operatorname{d}_{1,0}^{\gamma_{0}}$ and $\partial_{\mathcal{V}}:=\operatorname{d}_{0,1}^{\gamma_{0}}$ are the coboundary operators on $\Gamma(\wedge^{\bullet}T^{*}M)$ associated to the foliated differentials $\operatorname{d}_{\mathcal{F}}$ and $\operatorname{d}_{\mathcal{V}}$. So, $\partial_{\mathcal{F}}^{2}=0$, $\partial_{\mathcal{V}}^{2}=0$ and $\partial_{\mathcal{F}}\partial_{\mathcal{V}}+\partial_{\mathcal{V}}\partial_{\mathcal{F}}=0$.

Consider the subspaces $\mathcal{C}^{p}$ defined in \eqref{eq:dRCasim}. In particular, $\mathcal{C}^{0}=\operatorname{Casim}(M,P)$. Furthermore, because of \eqref{eq:FlatPoissConn}, $\mathcal{C}^{\bullet} := \bigoplus_{p\in\mathbb{Z}}\mathcal{C}^{p}$ is a $\partial_{\mathcal{F}}$-invariant subspace of $\Gamma(\wedge^{\bullet}T^{*}M)$ and hence the restriction $\overline{\partial}_{\mathcal{F}} := \partial_{\mathcal{F}}|_{\mathcal{C}^{\bullet}}$ is a well-defined coboundary operator. This gives rise to a cochain subcomplex $(\mathcal{C}^{\bullet},\overline{\partial}_{\mathcal{F}})$ of $(\Gamma(\wedge^{\bullet}\mathbb{V}^{0}),\partial_{\mathcal{F}})$ attributed to the flat Poisson foliation which will be called the \emph{foliated de Rham-Casimir complex} \cite{VeVo-16}. The corresponding cohomology space will be denoted by $H_{\overline{\partial}_{\mathcal{F}}}^{\bullet}$.

We have the following useful property \cite{VeVo-16}.

\begin{lemma}\label{lemma:Injec}
The natural homomorphism from $H_{\overline{\partial}_{\mathcal{F}}}^{1}$ to the first foliated de Rham cohomology $H_{\operatorname{dR}}^{1}(\mathcal{F})$ is injective if and only if
\begin{equation}\label{eq:Injec}
\partial_{\mathcal{F}}(\operatorname{Casim}(M,P))=\partial_{\mathcal{F}}(C^{\infty}(M))\cap\mathcal{C}^{1}.
\end{equation}
\end{lemma}

We say that a $\mathcal{V}$-coupling Poisson structure $\Pi$ on the flat Poisson foliation $(M,\mathcal{V},P,\mathcal{F})$ is \emph{compatible} if
$\Pi_{0,2}=P$ and the Poisson connection $\gamma$ induced by the normal subbundle $\mathbb{H}=\Pi^{\sharp}(T^{*}M)$ satisfies the condition
\[
\gamma_{0}(X)-\gamma(X)\text{ is tangent to }P^{\sharp}(T^{*}M),\quad\forall X\in\Gamma(TM).
\]
This compatibility condition implies that $\overline{\operatorname{d}}^{\gamma}=\overline{\partial}_{\mathcal{F}}$ and hence the cochain complex $(\mathcal{C}^{\bullet},\overline{\operatorname{d}}^{\gamma})$ associated to $\Pi$ coincides with the foliated de Rham-Casimir complex. Also, we say that $\Pi$ is \emph{strongly compatible} if there exists $\mu\in\Gamma(\mathbb{V}^{0})$ such that $\gamma$ and $\gamma_{0}$ are related by \eqref{eq:Gau3}.

First, we formulate a unimodularity criterion for the class of strongly compatible Poisson structures which involves the injectivity condition \eqref{eq:Injec}.

\begin{theorem}\label{teo:UnimodFlat}
Let $(M,\mathcal{V},P,\mathcal{F})$ be a flat Poisson foliation and $\Pi$ a strongly compatible coupling Poisson structure. If $\Pi$ is unimodular then
\begin{equation}\label{eq:UnimodFlat}
\operatorname{Mod}(\mathcal{F})=0 \qquad\text{and}\qquad \operatorname{Mod}(M,\mathcal{V},P)=0.
\end{equation}
Conversely, under the injectivity condition \eqref{eq:Injec}, the unimodularity of $(M,\Pi)$ is equivalent to \eqref{eq:UnimodFlat}.
\end{theorem}

\begin{proof}
Since $\Pi$ is compatible, we have $\overline{\operatorname{d}}^{\gamma}=\overline{\partial}_{\mathcal{F}}$, so $H^{1}_{\overline{\operatorname{d}}^{\gamma}}= H^{1}_{\overline{\partial}_{\mathcal{F}}}$. Moreover, if $\operatorname{Mod}(M,\mathcal{V},P)=0$, then the cohomology classes $[\theta^{\gamma}_{\tau}]\in H^{1}_{\overline{\operatorname{d}}^{\gamma}}$ and $[\theta^{\gamma_{0}}_{\tau}]\in H^{1}_{\overline{\partial}_{\mathcal{F}}}$ of the divergence 1-forms also coincide. Indeed, by the strong compatibility, formula \eqref{eq:Gau4} holds, so condition $\operatorname{Mod}(M,\mathcal{V},P)=0$ implies $[\theta^{\gamma}_{\tau}]=[\theta^{\gamma_{0}}_{\tau}]$. On the other hand, as shown in Section \ref{sec:OriFol}, the $\partial_{\mathcal{F}}$-cohomology class of $\theta^{\gamma_{0}}_{\tau}$ is the Reeb class $\operatorname{Mod}(\mathcal{F})\in H^{1}_{\operatorname{dR}}(\mathcal{F})$. In other words, the image of $[\theta^{\gamma_{0}}_{\tau}]$ under the morphism in Lemma \ref{lemma:Injec} is $\operatorname{Mod}(\mathcal{F})$. Finally, recall that by Corollary \ref{cor:Obstruction}, the unimodularity of $(M,\Pi)$ is equivalent to $\operatorname{Mod}(M,\mathcal{V},P)=0$ and $[\theta^{\gamma}_{\tau}]=0$. By our above discussion, this implies $\operatorname{Mod}(\mathcal{F})=0$. Conversely, under the injectivity condition \eqref{eq:Injec}, equation \eqref{eq:UnimodFlat} implies $[\theta^{\gamma}_{\tau}]=[\theta^{\gamma_{0}}_{\tau}]=0$. By Corollary \ref{cor:Obstruction}, the proof is complete.
\end{proof}

We have also the following unimodularity criterion in the case when the first cohomology of the foliated de Rham-Casimir complex is trivial.

\begin{theorem}\label{teo:Unimodular}
Let $\Pi$ be a compatible coupling Poisson structure on the flat Poisson foliation $(M,\mathcal{V},P,\mathcal{F})$. If
\begin{equation}\label{TRV1}
H_{\overline{\partial}_{\mathcal{F}}}^{1}=\{0\},
\end{equation}
then $(M,\Pi)$ is unimodular if and only if $\operatorname{Mod}(M,\mathcal{V},P)=0$.
\end{theorem}

\begin{proof}
By the compatibility condition, we have $H^{1}_{\overline{\operatorname{d}}^{\gamma}}= H^{1}_{\overline{\partial}_{\mathcal{F}}}$. Thus, the short exact sequence \eqref{eq:ShortExact} reads
\[
0\rightarrow H_{\overline{\partial}_{\mathcal{F}}}^{1} \overset{(\Pi_{H}^{\sharp})^{*}}{\longrightarrow} H_{\Pi}^{1}(M) \overset{\gamma^{*}}{\longrightarrow} \frac{\ker\rho}{\operatorname{Ham}(M,P)}
\rightarrow0.
\]
Hence, under condition \eqref{TRV1}, the projection $\gamma^{*}$ is an isomorphism. Moreover, by Corollary \ref{cor:ProyMod}, $\gamma^{*}$ maps $\operatorname{Mod}(M,\Pi)$ to $\operatorname{Mod}(M,\mathcal{V},P).$
\end{proof}

Now let us discuss some realizations of conditions \eqref{eq:Injec}, \eqref{TRV1}. Consider the space $\operatorname{Ham}(M,P)$ of Hamiltonian vector fields of the $\mathcal{V}$-tangent Poisson structure $P$. Then one can introduce the following two subspaces of $\operatorname{Ham}(M,P)$ depending on the foliation $\mathcal{F}$. Let $\operatorname{Ham}_{\mathcal{F}}(M,P) := \{P^{\sharp}\operatorname{d}f \mid f\in C^{\infty}(M), \partial_{\mathcal{F}}f=0\}$ be the Lie algebra of all Hamiltonian vector fields of $\mathcal{F}$-projectable functions, and $\operatorname{Ham}_{0}(M,P):=\{Y\in\operatorname{Ham}(M,P)\mid[Y,\Gamma_{\mathrm{pr}}(\mathbb{F})]=0\}$ the Lie algebra of $\mathcal{F}$-projectable Hamiltonian vector fields. It follows from $\Gamma_{\mathrm{pr}}(\mathbb{F})\subset\operatorname{Poiss}(M,P)$ that $\operatorname{Ham}_{\mathcal{F}}(M,P)\subseteq\operatorname{Ham}_{0}(M,P)$. Then, we have the following fact \cite{VeVo-16}: injectivity condition \eqref{eq:Injec} holds if and only if $\operatorname{Ham}_{\mathcal{F}}(M,P)=\operatorname{Ham}_{0}(M,P)$. This condition together with the assumption $H_{\operatorname{dR}}^{1}(\mathcal{F})=\{0\}$ on the triviality of the first foliated de Rham cohomology of $(M,\mathcal{F})$ implies \eqref{TRV1}.

Moreover, we have the following realization of condition \eqref{TRV1} in the case of a flat Poisson fibration. Suppose we have a transversal bi-fibration $N\overset{\nu}{\leftarrow}M\overset{\pi}{\rightarrow}S$,
\[
TM=\ker d\nu\oplus\ker d\pi.
\]
Let $\mathcal{F}=\{\nu^{-1}(\xi)\}_{\xi\in N}$ and $\mathcal{V}=\{\pi^{-1}(q)\}_{q\in S}$ be the regular foliations of $M$ defined by the fibers of
the submersions $\nu$ and $\pi$, respectively. So, $\mathbb{F}=T\mathcal{F}=\ker d\nu$ and $\mathbb{V}=T\mathcal{V}=\ker d\pi$. Assume also that we are given a Poisson tensor $P\in\Gamma(\wedge^{2}\mathbb{V})$ such that the triple $(M\overset{\pi}{\rightarrow}S,P,\mathcal{F})$ is a flat Poisson fibration, that is, $\Gamma_{\mathrm{pr}}(\mathbb{F})\subset\operatorname{Poiss}(M,P)$. Then, there exists a unique Poisson structure $\Psi$ on $N$ such that the projection $\nu:(M,P)\rightarrow(N,\Psi)$ is a Poisson map. One can show \cite{VeVo-16} that condition \eqref{TRV1} holds if
\[
H_{\operatorname{dR}}^{1}(\mathcal{F})=\{0\}\text{ and }H_{\Psi}^{1}(N)=\{0\}.
\]
Notice that the last condition implies \eqref{eq:Injec}.

We conclude this section with the construction of a class of unimodular compatible Poisson structures.

\paragraph{Flat Coupling Poisson Structures.} Let $(M,\mathcal{V},P,\mathcal{F})$ be a flat Poisson foliation. Suppose we are given a $\overline{\partial}_{\mathcal{F}}$-closed, $\mathbb{F}$-nondegenerated 2-form $\sigma_{0}\in\mathcal{C}^{2}$, that is, $\overline{\partial}_{\mathcal{F}}\sigma_{0}=0$ and
$\sigma_{0}^{\flat}|_{\mathbb{F}}:\mathbb{F}\rightarrow\mathbb{V}^{0}$ is an isomorphism. Then, one can define a coupling Poisson structure associated to the geometric data $(\sigma_{0},\gamma_{0},P)$:
\begin{equation}\label{FLP}
\Pi_{\operatorname{flat}}=\Pi_{F}+P,
\end{equation}
where $\Pi_{F}\in\Gamma(\wedge^{2}\mathbb{F})$ is a bivector field defined by the condition that the restriction $(\Pi_{F})^{\sharp}|_{\mathbb{V}^{0}}$ equals to the inverse of $-\sigma_{0}^{\flat}|_{\mathbb{F}}$. In this case, $\Pi_{F}$ is a regular Poisson tensor which together with $P$ forms a Poisson pair. Since the symplectic foliation of $\Pi_{F}$ is just $\mathcal{F}$, it is clear that $\Pi_{\operatorname{flat}}$ is a compatible Poisson structure and $\operatorname{Mod}(M,\Pi_{F})=-\Pi_{F}^{\sharp}(\operatorname{Mod}(\mathcal{F}))$. Assuming that $\mathcal{V}$ is orientable and equipped with a nowhere vanishing section $\tau\in\Gamma(\wedge^{\operatorname{top}}\mathbb{V}^{*})$, we define a volume form as $\Omega_{0}=\sigma_{0}^{l}\wedge\tau_{\gamma_{0}}$, $2l:=\operatorname{rank}\mathbb{F}$. Then, the modular vector field of $\Pi_{\operatorname{flat}}$ relative to $\Omega_{0}$ is represented as $Z_{\Pi_{\operatorname{flat}}}^{\Omega_{0}}=Z_{\Pi_{F}}^{\Omega_{0}}+$ $Z_{P}^{\Omega_{0}}$. Under the injectivity condition \eqref{eq:Injec}, we conclude from \eqref{eq:ModsFlat} and Theorem \ref{teo:UnimodFlat} that $\Pi_{\operatorname{flat}}$ is unimodular if and only if $\operatorname{Mod}(\mathcal{F})=0$ and $\operatorname{Mod}(M,P)=0$. In this case, according to Proposition \ref{prop:Gauge}, a gauge transformation \eqref{eq:Gau1}, \eqref{eq:Gau2} modifies $\Pi_{\operatorname{flat}}$ preserving the unimodularity property.

\section{Coupling Neighborhoods of a Symplectic Leaf}\label{sec:Leaf}

Let $(M,\Pi)$ be a Poisson manifold and $\iota:S\hookrightarrow M$ an embedded symplectic leaf.

By a \emph{coupling neighborhood} of $S$, we mean an open neighborhood $N$ of $S$ in $M$ equipped with a surjective submersion $\pi:N\rightarrow S$ such that $\pi\circ\iota=\operatorname{Id}_{S}$,
\begin{equation}\label{eq:CoupNeigh}
\operatorname{rank}\mathbb{H}=\dim S,\qquad\text{and}\qquad\mathbb{H}\cap\mathbb{V}=\{0\},
\end{equation}
where $\mathbb{V}=\ker d\pi$ is the vertical subbundle of $\pi$ and $\mathbb{H}=\Pi^{\sharp}(\mathbb{V}^{0})$ is the horizontal subbundle
associated to $\Pi$. It is clear that conditions in \eqref{eq:CoupNeigh} are equivalent to the splitting $TN=\mathbb{H}\oplus\mathbb{V}$. Therefore the restriction $\Pi|_{N}$ is a $\mathcal{V}$-coupling Poisson structure on $N$, where the foliation $\mathcal{V}=\mathcal{V}^{\pi}:=\{N_{q}=\pi^{-1}(q)\}_{q\in S}$ is given by the $\pi$-fibers. Taking into account that the symplectic leaf $S$ is an orientable manifold, we conclude that the Reeb class of $\mathcal{V}^{\pi}$ is trivial (see Example \ref{ex:FiberReeb}). So, $\Pi|_{N}$ has a bigraded decomposition into a horizontal part of constant rank and a vertical Poisson tensor $P\in\Gamma(\wedge^{2}\mathbb{V})$ vanishing at the points of $S$. The Poisson structure $P$ is said to be a \emph{transverse Poisson structure} of the leaf. The restriction $P_{q}:=P|_{N_{q}}$ of $P$ to the fiber $N_{q}$ over every point $q\in S$ just gives the transverse Poisson structure of $q$ due to Weinstein's splitting theorem \cite{We-83}. Moreover, the Poisson connection $\gamma$ on $N$ is defined by $\ker\gamma=\Pi^{\sharp}(\mathbb{V}^{0})$ and the coupling form $\sigma\in\Gamma(\wedge^{2}\mathbb{V}^{0})$ has the representation $\sigma=\pi^{*}\omega+\tilde{\sigma}$, where $\omega$ is the symplectic form on $S$ and $\tilde{\sigma}$ is a horizontal 2-form vanishing at $S$.

For a given embedded symplectic leaf $S$, there exists always such a coupling neighborhood $N$ \cite{Vo-01}. In particular, one can choose $N$ as a tubular neighborhood of $S$ diffeomorphic to the normal bundle $E=T_{S}M\diagup TS$ of the symplectic leaf. If the normal bundle $E$ is orientable, then $N$ admits a volume form. Of course, this is true in the case when $M$ is orientable. Hence, under the orientability hypothesis, the point is to study the germs at $S$ of the modular vector fields of $\Pi$ and the corresponding germified modular class.

First we formulate the following result.

\begin{proposition}\label{prop:LeafMods}
If the Poisson structure $\Pi$ is unimodular in a neighborhood of the embedded symplectic leaf $S$, then there exists a coupling neighborhood $N$ of $S$ such that the transverse Poisson structure $P$ of $S$ is also unimodular.
\end{proposition}

\begin{proof}
The statement follows from Lemma \ref{lemma:Unimods}, Proposition \ref{prop:ModsAndReeb} and the fact that in a tubular neighborhood of $S$, the Reeb class of the fibration is trivial.
\end{proof}

We say that the germ of the transverse Poisson structure at a point $q\in S$ is unimodular if there exists a submanifold $N_{q}$ of $M$ meeting the symplectic leaf $S$ at $q$ transversally, and such that
\[
\operatorname{Mod}(N_{q},P_{q})=0.
\]

\begin{theorem}\label{teo:Unimodular3}
Let $S$ be an embedded symplectic leaf of an orientable Poisson manifold $(M,\Pi)$ and $q\in S$ a fixed point. Assume that the germ at $q\in S$ of the transverse Poisson structure $P_{q}$ is unimodular. Then, one can choose a coupling neighborhood $(N\overset{\pi}{\rightarrow}S)$ of $S$ with properties: there exists a leaf-wise volume form $\tau\in\Gamma(\wedge^{\operatorname{top}}\mathbb{V}^{*})$ of the vertical subbundle $\mathcal{V}^{\pi}$ such that the modular vector field of the Poisson foliation $(N,\mathcal{V}^{\pi},P)$ vanishes. Furthermore, the modular vector field of $\Pi|_{N}$ with respect to the volume form $\Omega=\sigma^{l}\wedge\tau_{\gamma}$ is tangent to the symplectic foliation and the corresponding modular class is given by
\[
\operatorname{Mod}(N,\Pi|_{N})=-(\Pi_{H}^{\sharp})^{*}[\theta_{\tau}^{\gamma}].
\]
Here, the divergence form $\theta_{\tau}^{\gamma}$ induced by the pair $(\tau,\gamma)$ is a 1-cocycle of the cochain complex $(\mathcal{C}^{\bullet}, \overline{\operatorname{d}}^{\gamma})$.
\end{theorem}

\begin{proof}
Choose a coupling neighborhood $N$ such that the Poisson fiber bundle $(N\overset{\pi}{\rightarrow}S,P)$ is locally trivial with typical fiber $(N_{q},P_{q})$. Then, by the proof of Corollary \ref{cor:LocallyTrivialMod} we conclude that $\operatorname{Mod}(N,\mathcal{V}^{\pi},P)=0$. From Theorem \ref{teo:Class}, we derive the desired result.
\end{proof}

\paragraph{Flat Coupling Neighborhoods.} We say that a coupling neighborhood $N\overset{\pi}{\rightarrow}S$ over the leaf $S$ is \emph{flat} if there exists a regular foliation $\mathcal{F}$ on $N$ such that (i) the tangent bundle $\mathbb{F}:=T\mathcal{F}$ is complementary to the vertical subbundle $\mathbb{V}$ of $\pi$; (ii) each $\pi$-projectable section of $\mathbb{F}$ is a Poisson vector field (an infinitesimal automorphism) of the transverse Poisson structure $P\in\Gamma(\wedge^{2}\mathbb{V})$; (iii) the foliation is compatible with the Poisson connection $\gamma$ on $N$ associated to the horizontal subbundle $\mathbb{H}$ in the following sense:
\[
X-\gamma(X)\text{ is tangent to }P^{\sharp}(T^{*}M) \qquad\forall X\in\Gamma_{\mathrm{pr}}(\mathbb{F}).
\]

\begin{theorem}\label{teo:Unimodular2}
Let $S$ be an embedded symplectic leaf of an orientable Poisson manifold $(M,\Pi)$ which admits a flat coupling neighborhood $(N\overset{\pi}{\rightarrow}S,\mathcal{F})$. Let $P$ be the transverse Poisson structure on $N$ of the leaf. If $H_{\overline{\partial}_{\mathcal{F}}}^{1}=\{0\}$, then the following assertions are equivalent:
\begin{itemize}
\item[(a)] the restriction of $\Pi$ to $N$ is unimodular;
\item[(b)] the Poisson manifold $(N,P)$ is unimodular;
\item[(c)] the Poisson fibration $(N\overset{\pi}{\rightarrow}S,P)$ is unimodular,
\begin{equation}\label{CON1}
\operatorname{Mod}(N,\mathcal{V}^{\pi},P)=0.
\end{equation}
\end{itemize}
\end{theorem}

\begin{proof}
By Theorem \ref{teo:Unimodular}, the assertions of items (a) and (c) are equivalent. The equivalence of (b) and (c) follows from Proposition \ref{prop:ModsAndReeb} and the orientability of the symplectic leaf $S$.
\end{proof}

Suppose we are given a flat Poisson fiber bundle $(\pi:M\rightarrow S,P,\mathcal{F})$ over a symplectic base $(S,\omega)$. Assume that $S$ is an embedded submanifold of $M$, the inclusion map $\iota:S\hookrightarrow M$ is a section of $\pi$, $T_{S}\mathcal{F}=TS$ and the vertical Poisson structure $P\in\Gamma(\wedge^{2}\mathbb{V})$ vanishes at the points of $S$. Let $\gamma_{0}$ be the flat Poisson connection on the Poisson fiber bundle $(\pi:M\rightarrow S,P)$ associated to the foliation $\mathcal{F}$. Denote by $\operatorname{hor}^{\gamma_{0}}$ the corresponding $\gamma_{0}$-horizontal lift and by $\psi\in\Gamma(\wedge^{2}TS)$ the nondegenerated Poisson tensor of the symplectic manifold $S$. Then, putting $\sigma_{0}=\pi^{*}\omega$, we get that formula \eqref{FLP} gives the following flat coupling Poisson tensor on $M$: $\Pi_{\operatorname{flat}}=\operatorname{hor}^{\gamma_{0}}(\psi)+P$. It is clear that $(S,\omega)$ is a symplectic leaf of $\Pi_{\operatorname{flat}}$. Moreover, for a given horizontal 1-form $\mu\in\Gamma(\mathbb{V}^{0})$ on $M$ vanishing along $S$, $\iota^{*}\mu=0$, there exists a neighborhood $N$ of $S$ in $M$, such that the gauge transformation \eqref{eq:Gau1}, \eqref{eq:Gau2} associated to $\mu$ is well-defined. Therefore, $N$ is a flat coupling neighborhood of $S$ for the deformed Poisson structure $\widetilde{\Pi}_{\operatorname{flat}}$. We get from Theorem \ref{teo:UnimodFlat} that the injectivity condition \eqref{eq:Injec} together with $\operatorname{Mod}(\mathcal{F})=0$ and \eqref{CON1} provides the unimodularity of $\widetilde{\Pi}_{\operatorname{flat}}$ (see, Section \ref{sec:Flat}).

\bigskip\noindent\textbf{Acknowledgements.} {We are very grateful to Misael Avenda\~no-Camacho, Rub\'en Flores-Espinoza, and Jos\'e C. Ru\'iz-Pantale\'on for several illuminating discussions and comments on this work. We are also grateful to an anonymous Referee for the helpful remarks which improved the presentation and the content of the manuscript. E.V.-B. and Yu.V. have been supported by Consejo Nacional de Ciencia y Tecnolog\'ia (CONACYT) under the research grant 219631.

\bibliographystyle{acm}
\bibliography{UnimodPoissFol-CorrectedVersion.bbl}
\end{document}